\documentclass{amsart}
\usepackage{amsmath}
\usepackage{amssymb}
\usepackage{hyperref}
\usepackage{esint}
\usepackage{color}
\usepackage{mathrsfs}

\newcommand{\Om}{\Omega}
\newcommand{\lb}{\lambda}
\newcommand{\ve}{\varepsilon}
\newcommand{\sg}{\sigma}

\newcommand{\fun}[6]{\ensuremath{\Phi^{#1,#2}_{#3,#4}(#6,#5)}}

\makeatletter
\providecommand*{\dif}%
   {\@ifnextchar^{\DIfF}{\DIfF^{}}}
\def\DIfF^#1{%
   \mathop{\mathrm{\mathstrut d}}%
      \nolimits^{#1}\gobblespace
}
\def\gobblespace{%
   \futurelet\diffarg\opspace}
\def\opspace{%
   \let\DiffSpace\!%
   \ifx\diffarg(%
      \let\DiffSpace\relax
   \else
      \ifx\diffarg\[%
         \let\DiffSpace\relax
      \else
         \ifx\diffarg\{%
            \let\DiffSpace\relax
         \fi\fi\fi\DiffSpace}

\newcommand{\medxr}{\mu                                               B(x,r)}
\newcommand{\triplet}{\ensuremath{(X,d,\mu)} }

\theoremstyle{plain}
\newtheorem{theorem}{Theorem}[section]
\newtheorem*{theorema}{Theorem A}
\newtheorem{lemma}[theorem]{Lemma}
\newtheorem{proposition}[theorem]{Proposition}
\newtheorem{corollary}[theorem]{Corollary}

\theoremstyle{definition}
\newtheorem{definition}[theorem]{Definition}

\theoremstyle{remark}
\newtheorem{remark}[theorem]{Remark}

\begin{document}

\title[Riesz type potential operators in generalized grand Morrey spaces]{Riesz type potential operators in  \\generalized grand Morrey spaces}

\author[V. Kokilashvili]{Vakhtang Kokilashvili}
\address{Department of Mathematical Analysis, A. Razmadze Mathematical Institute, I. Java\-khishvili Tbilisi State University,
2. University Str., 0186 Tbilisi, Georgia}
\address{International Black Sea University, 3 Agmashenebeli Ave., Tbilisi 0131, Georgia } 
\email{kokil@rmi.ge}

\author[A. Meskhi]{Alexander Meskhi}
\address{Department of Mathematical Analysis, A. Razmadze Mathematical Institute, 2.  University Str., 0186 Tbilisi, Georgia}
\address{Department of Mathematics,  Faculty of Informatics and Control
Systems,  Georgian Technical University, 77, Kostava St., Tbilisi, Georgia}
\address{Abdus Salam School of Mathematical Sciences, GC
University, 68-B New Muslim Town, Lahore, Pakistan}
\email{meskhi@rmi.ge}

\author[H. Rafeiro]{Humberto Rafeiro}
\address{Instituto Superior T\'ecnico\\ Departamento de Matem\'atica \\Centro CEAF, Av. Rovisco Pais \\1049--001 Lisboa\\Portugal}  
\address{Pontificia Universidad Javeriana\\ Departamento de Matem\'aticas \\Cra 7a No 43-82 Ed. Carlos Ortiz 604\\ Bogot\'a, \textsc{Colombia} }
\email{hrafeiro@math.ist.utl.pt; silva-h@javeriana.edu.co}


\begin{abstract}
In this paper we introduce generalized grand Morrey spaces in the framework of quasimetric measure spaces, in the spirit of the so-called grand Lebesgue spaces. We prove a kind of \textit{reduction  lemma} which is applicable to  a variety of operators to reduce their boundedness in generalized grand Morrey spaces to the corresponding  boundedness in  Morrey spaces,  as a result of this application, we obtain the boundedness of the Hardy-Littlewood maximal operator  as well as the boundedness of Calder\'on-Zygmund potential type operators. Boundedness of Riesz type potential operators are also obtained in the framework of homogeneous and also in the nonhomogeneous case in generalized grand Morrey spaces.
\end{abstract}

\maketitle

\tableofcontents

\section{Introduction}

In 1992 T. Iwaniec and C. Sbordone \cite{iwa_sbor1992},  in their studies related with the integrability properties of the Jacobian in a bounded open set $\Omega,$ introduced a new type of function spaces  $L^{p)}(\Omega),$   called \textit{grand Lebesgue spaces}. A generalized version of them, $L^{p),\theta}(\Omega)$ appeared in  L. Greco, T. Iwaniec and C. Sbordone \cite{greco}.
 Harmonic analysis related to these spaces and  their associate spaces (called \textit{small Lebesgue spaces}), was intensively studied during last years due to various applications, we mention e.g. \cite{capone_fio, fratta_fio, fio200, fio_gupt_jain, fio_kara, fio_rako, kok_2010}.
Recently in \cite{samko_umar} there was  introduced a version of weighted grand Lebesgue spaces adjusted for sets $\Om\subseteq \mathbb{R}^n$  of infinite measure, where  the integrability of
$|f(x)|^{p-\ve}$ at infinity was controlled by means of a weight, and there generalized grand Lebesgue spaces were also considered, together with the study of classical operators of harmonic analysis in such spaces.  Another idea of introducing ``bilateral" grand Lebesgue spaces on sets of infinite measure was suggested in \cite{357ad}, where the structure of such spaces was investigated, not operators; the spaces in \cite{357ad} are two parametrical with respect to the exponent $p$, with the norm involving $\sup_{p_1<p<p_2}.$ \\

Morrey spaces $L^{p,\lambda}$ were introduced in  1938 by C. Morrey \cite{405a}  in relation to regularity problems of solutions to partial differential equations, and provided a  useful tool in the regularity theory of PDE's (for  Morrey spaces we refer to books \cite{187a, kuf}, see also  \cite{rafsamsam} where an overview of various generalizations may be found).
 
 Recently, in the spirit of grand Lebesgue spaces, A. Meskhi \cite{meskhi2009, meskhi} introduced \textit{grand Morrey spaces}  (in \cite{meskhi2009} it was already defined on quasi-metric measure spaces with doubling measure) and obtained results on the boundedness   of the maximal operator, Calder\'on-Zygmund singular operators and Riesz potentials. The boundedness of commutators of singular and potential operators in grand Morrey spaces was already treated by X. Ye \cite{ye_xf}.  Note that  the ``\textit{grandification} procedure" was applied only to the parameter $p.$ \\

In this paper we make a further step  and apply the ``grandification procedure" to both the parameters, $p$ and $\lambda,$ obtaining \textit{generalized grand Morrey spaces} $L^{p),\lambda)}_{\theta,A}(X,\mu).$ In this new framework we obtain a reduction boundedness theorem, which reduces the boundedness of operators (not necessarily linear ones) in generalized grand Morrey spaces to the corresponding boundedness in classical Morrey spaces. \\

In our future investigations we plan to establish the boundedness of commutators of singular and fractional integrals and its applications, for example, in regularity problems for  the solution of elliptic equations in non-divergence form from generalized Morrey spaces viewpoint. \\

\noindent \textbf{Notation:}

\noindent    $d_X$    denotes    the    diameter    of    the    $X$   set;\\
\noindent     $A     \sim     B$     for     positive     $A$     and     $B$
means  that  there  exists  $c>0$  such  that  $c^{-1}A  \leqslant B \leqslant c A$;\\
\noindent$B(x,r)=\{y\in             X:             d(x,y)<r            \}$;\\
\noindent by $c$ and $C$ we denote various absolute positive constants,
which may have different values even in the same line;\\
$\hookrightarrow$ means continuous imbedding;\\
\noindent $\fint_B f \dif \mu$ denotes the integral average of $f$, i.e. $\fint_B f \dif \mu:= \frac{1}{\mu B} \int_B f \dif \mu$;\\
\noindent $p^\prime$ stands for the conjugate exponent $1/p+1/p^\prime=1$.

\section{Preliminaries}

\subsection{Spaces of homogeneous type}\label{preliminaries}
Let $X:=\triplet$ be a topological space with a complete measure $\mu$ such that the space of compactly supported continuous functions is dense in $L^1(X,\mu)$ and $d$ is a quasimetric, i.e.  it is  a non-negative real-valued function $d$ on $X\times X$ which  satisfies the conditions:
\begin{enumerate}
\item[(i)] $d(x,y)=0$ if and only if $x=y$;
\item[(ii)] there exists a constant $C_t>0$ such that $d(x,y)\leqslant  C_t  [d(x,z)+d(z,y)]$, for all $x,y,z \in X$, and
\item[(iii)] there exists a constant $C_s>0$ such that $d(x,y)\leqslant  C_s \cdot d(y,x)$, for all $x,y \in X$.
\end{enumerate}
Let  $\mu$  be  a  positive measure on the $\sigma$-algebra of subsets of $X$ which  contains  the  $d$-balls $B(x,r).$
Everywhere in the sequel we assume that all the balls have a finite measure, that is, $\mu B(x,r)<\infty$ for all
$x\in X$ and $r>0$ and that for every neighborhood $V$ of $x\in X$, there exists $r>0$ such that $B(x,r)\subset V$.

We  say  that the measure $\mu$ is \textit{lower  $\alpha$-Ahlfors regular},
if
\begin{equation}\label{lowerahlforscondition}
 \medxr\geq                                                         cr^\alpha
\end{equation}
and \textit{upper  $\beta$-Ahlfors regular} (or, it satisfies the  \textit{growth condition of degree $\beta$}), if
\begin{equation}\label{upperahlforscondition}
 \mu B(x,r)\le cr^\beta,
\end{equation}
where  $\alpha,\beta, c >0$ does not depend on $x$ and $r$. 
 When $\alpha=\beta$,
the measure  $\mu$ is simply called \textit{$\alpha$-Ahlfors regular}.\\

The  condition
\begin{equation}\label{doublingcondition}
 \mu    B(x,2r)\leqslant   C_d \cdot \medxr,   \quad   C_d   >   1
\end{equation}
on the measure $\mu$ with $C_d $ not depending on $x\in X$ and
$0<r<d_X$, is known as the \textit{doubling condition}. 

Iterating     it,      we      obtain
\begin{equation}\label{doublingcondition2}
 \frac{\mu B(x,R)}{\mu B(y,r)} \leqslant C_d \left( \frac{R}{r}\right)^{\log_2 C_d}, \quad  0< r\leqslant R
\end{equation}
for  all  $d$-balls  $B(x,R)$ and $B(y,r)$ with $B(y,r)\subset B(x,R)$.

The triplet $(X,d,\mu)$, with  $\mu$ satisfying the doubling condition, is
said a  \textit{space  of  homogeneous  type}, abbreviated from now on simply as SHT. For some important examples of an SHT we refer e.g. to \cite{coifmanweiss}.

From  \eqref{doublingcondition2}   it   follows   that every
 homogeneous  type space \triplet with
 a  finite  measure is lower  $(\log_2 C_d)$-Ahlfors regular.

\subsection{Grand Lebesgue spaces} For $1<p<\infty$, $\theta >0$ and $0<\ve<p-1$ the \emph{grand Lebesgue space} is the set of measurable functions for which
\begin{equation}\label{added}
 \|f\|_{L^{p),\theta}(X,\mu)}:=\sup_{0<\varepsilon<p-1} \varepsilon^\frac{\theta}{p-\varepsilon} \|f\|_{L^{p-\varepsilon}(X,\mu)}<\infty
 \end{equation}
where $\|f\|^p_{L^p(X,\mu)}:=\int_X |f(y)|^p \dif \mu(y)$. 
In the case $\theta=1,$ we denote $L^{p),\theta}(X,\mu):=L^{p)}(X,\mu).$

When $\mu X<\infty$, then for all $0<\ve\leqslant p-1$ we have
\[
L^p(X,\mu)\hookrightarrow L^{p)}(X,\mu) \hookrightarrow L^{p-\ve}(X,\mu).
\]

For more properties of grand Lebesgue spaces, see \cite{kok_2010}.

\subsection{Morrey spaces}
For $1\leqslant p< \infty $ and $0\leqslant \lambda <1$, the usual Morrey space $L^{p,\lambda}(X,\mu)$ is introduced as the set of all measurable functions such that
\begin{equation}\label{eq:morrey_norm}
\|f\|_{L^{p,\lambda}(X,\mu)}:=\sup_{\stackrel{x \in X}{0<r< d_X }}\left(\frac{1}{\medxr ^{\lambda}}   \int_{B(x,r)} |f(y)|^{p} \dif \mu(y)\right)^\frac{1}{p}<\infty.
\end{equation}

Sometimes we will need a modification of the previous Morrey space, namely,  we define $\mathscr L^{p,\lambda}(X,\mu)$  as 
\begin{equation}\label{eq:morrey_new_norm}
\|f\|_{\mathscr L^{p,\lambda}(X,\mu)}:=\sup_{\stackrel{x\in X}{0<r<d_X}} \left (  \frac{1}{r^{\gamma \lambda}}  \int_{B(x,r)} |f(y)|^p \dif \mu(y) \right)^\frac{1}{p}<\infty.
\end{equation}

\section{Generalized grand Morrey spaces and the reduction lemma}
In this section we will assume that the measure $\mu$ is upper $\gamma$-Ahlfors regular.
After introducing generalized grand Morrey spaces in the framework of SHT in a slightly more general way as was done for the Euclidean case in H. Rafeiro \cite{rafeiro2012}, we show that the same reduction lemma is valid in the setting of SHT.

We introduce the following functional
 \begin{equation}\label{added}
 \Phi^{p,\lb}_{\varphi,A}(f,s):=\sup_{0<\varepsilon<s} \varphi(\varepsilon)^\frac{1}{p-\varepsilon} \|f\|_{L^{p-\varepsilon,\lambda -A(\varepsilon)}(X,\mu)}.
 \end{equation}

\begin{definition}[Generalized grand Morrey spaces]\label{def:ggms}
Let  $1<p<\infty$, $0\leqslant \lambda <1$, $\varphi$ be a positive bounded function with $\lim_{t \to 0+} \varphi(t)=0$ and $A$ be a non-decreasing real-valued non-negative function with $\lim_{x\to 0+} A(x)=0$. By
 $L^{p),\lb)}_{\varphi,A}(X,\mu)$ we denote the space of measurable functions having  the finite norm
\begin{equation}\label{equ:norm}
	\|f\|_{L^{p),\lb)}_{\varphi,A}(X)}:=
\Phi^{p,\lb}_{\varphi,A}(f,s_{\max}), \quad s_{\max}= \min\left\{p-1,a\right\}
\end{equation}
where $a=\sup \{x>0: A(x)\leqslant \lambda\}$.
\end{definition}

\begin{remark}\label{rem:quotient_morrey}
For appropriate $\varphi$, in the case $A\equiv 0, \lambda>0$ we recover  the Grand Morrey spaces introduced in  A. Meskhi \cite{meskhi}, and when  $\lambda=0$, $A\equiv 0$ we have the   grand Lebesgue spaces introduced in \cite{greco} (and in \cite{iwa_sbor1992} in the case $\theta=1$).
\end{remark}

For fixed $p, \lambda, \varphi, A, f$ we have that $ s \mapsto \Phi^{p,\lb}_{\varphi,A}(f,s)$ is a non-decreasing function, but  it is possible to  estimate $\Phi^{p,\lb}_{\varphi,A}(f,s)$ via $\Phi^{p,\lb}_{\varphi,A}(f,\sg)$ with $\sigma<s$ as follows.

\begin{lemma}\label{lem:dominance}  For $0<\sigma<s<s_{\max}$ we have that
\begin{equation}\label{equ:dominance}
	\Phi^{p,\lb}_{\varphi,A}(f,s) \leqslant C \varphi(\sigma)^{-\frac{1}{p-\sigma}} \Phi^{p,\lb}_{\varphi,A}(f,\sg),
\end{equation}
where $C$ depends on   $\gamma,$ the parameters $p,\lb, \varphi, A$  and the diameter $d_X,$ but does not depend on $f, s$ and $\sg$.
\end{lemma}

\begin{proof}
For fixed $\sigma$ and $0<\sigma<s<s_{\max}$ we have
\begin{equation}\label{equ:break}
\Phi^{p,\lb}_{\varphi,A}(f,s)= \max \Big\{\Phi^{p,\lb}_{\varphi,A}(f,\sg) , \underbrace{\sup_{\sigma \leqslant \varepsilon < s} \varphi(\varepsilon)^\frac{1}{p-\varepsilon} \|f\|_{L^{p-\varepsilon, \lambda-A(\varepsilon)}(X)}}_{=:I} \Big\}.
\end{equation}
To estimate
\[
I=\sup\limits_{\sigma \leqslant \varepsilon < s} \varphi(\varepsilon)^\frac{1}{p-\varepsilon}
 \sup\limits_{\stackrel{x \in X}{0<r\leqslant d_X}} \mu B(x,r)^{\frac{A(\ve)-\lambda}{p-\ve}} \|f\|_{L^{p-\varepsilon}(B(x,r))},
 \]
since  $\varphi$ is a positive bounded function and by Holder's inequality,  
we have
  \[
\begin{split}	I &\leqslant C \sup_{0<\eta<s_{\max}}\varphi(\eta)^\frac{1}{p-\eta} \sup_{\sigma < \varepsilon < s_{\max}}\sup_{\stackrel{x \in X}{0<r\leqslant d_X}} \mu B(x,r)^{\frac{1+A(\ve)-\lambda}{p-\varepsilon} } \times \\
&\hspace{7cm}\left(  \fint_{B(x,r)} |f(y)|^{p-\sigma} \dif \mu(y) \right)^\frac{1}{p-\sigma}\\
	&\leqslant C \sup_{\sigma < \varepsilon <s_{\max}}\sup_{\stackrel{x \in X}{0<r\leqslant d_X}} \mu B(x,r)^{\Delta(\ve,\sigma)} \times \\
&\hspace{4.5cm}\left( \frac{\varphi(\sigma) \varphi(\sigma)^{-1}}{\mu B(x,r)^{\lambda-A(\sigma)}}  \int_{B(x,r)} |f(y)|^{p-\sigma} \dif \mu(y) \right)^\frac{1}{p-\sigma}
\end{split}
\]
with $\Delta(\ve,\sigma):= \frac{1+A(\ve)-\lambda}{p-\ve}-\frac{1+A(\sigma)-\lambda}{p-\sigma}.$

Observe that, since $A$ is non-decreasing we have that
\[
\begin{split}
\Delta(\ve, \sigma)&=\frac{1+A(\ve)-\lambda}{p-\ve} -  \frac{1+A(\sigma)-\lambda}{p-\sigma}\\
&=\frac{(\ve-\sigma)(1-\lambda)+A(\ve)(p-\sigma)-A(\sigma)(p-\ve)}{(p-\sg)(p-\ve)}\geqslant 0,
\end{split}
\]
and for $0\leqslant \ve  <s_{\max}$ we have $\frac{1-\lambda}{p} \leqslant \frac{1+A(\ve) -\lambda}{p-\ve} \leqslant 1$, so that $0\leqslant \Delta(\ve,\sigma)\leqslant 1.$ Then 
\[
\mu B(x,r)^{\frac{1+A(\ve)-\lambda}{p-\ve} -  \frac{1+A(\sigma)-\lambda}{p-\sigma}}\leqslant C \max\{1, d_X^\gamma \}
\]
since $\mu$ satisfies the $\gamma$-growth condition,  and we obtain
\[
\begin{split}
	I&\leqslant C   \sup_{\stackrel{x \in X}{0<r\leqslant d_X}}\varphi(\sigma)^{-\frac{1}{p-\sigma}} \left( \frac{\varphi(\sigma)}{\mu B(x,r)^{\lambda-A(\sigma)}} \int_{B(x,r)} |f(y)|^{p-\sigma} \dif \mu(y) \right)^\frac{1}{p-\sigma}\\
&\leqslant C    \varphi(\sigma)^{-\frac{1}{p-\sigma}}  \sup_{0<\varepsilon \leqslant \sigma}  \varphi(\varepsilon)^\frac{1}{p-\varepsilon} \|f\|_{L^{p-\varepsilon,\lambda-A(\varepsilon)}(X)}\\
&= C   \varphi(\sigma)^{-\frac{1}{p-\sigma}}  \cdot \Phi^{p,\lb}_{\varphi,A}(f,\sg). \qedhere
\end{split}
\]
\end{proof}

From Lemma \ref{lem:dominance} we immediately have
\begin{lemma}\label{lem:dominance1}
For $0<\sigma<s_{\max}$, the norm defined in \eqref{equ:norm} has the following dominant
\begin{equation}\label{equ:dominant}
	\|f\|_{L^{p),\lb)}_{\varphi,A}(X)}\leqslant C \frac{ \fun{p}{\lambda}{\varphi}{A}{\sigma}{f}}{\varphi(\sigma)^\frac{1}{p-\sigma}},
\end{equation}
\end{lemma}

\begin{lemma}[Reduction  lemma]\label{main} Let $U$ be an operator (not necessarily sublinear) bounded in the Morrey spaces
	\begin{equation}\label{eq:boun_classical}
\|Uf\|_{L^{q-\ve,\lambda-A_2(\ve)}(X)} \leqslant 	C_{p-\ve,\lambda-A_1(\ve),q-\ve,\lambda-A_2(\ve)} \|f\|_{L^{p-\ve,\lambda-A_1(\ve)}(X)}
\end{equation}
for all sufficiently small $\ve\in (0,\sg]$, where
 $0<\sigma<s_{\max}.$ If  
\begin{equation}\label{reduction_condition}
\sup_{0<\varepsilon< \sigma} C_{p-\ve,\lambda-A_1(\ve),q-\ve,\lambda-A_2(\ve)} <\infty
\end{equation}
and
\begin{equation}\label{eq:finite_ratio}
\sup_{0<\ve<\sigma} \frac{\psi(\ve)^{\frac{1}{q-\ve}}}{\varphi(\ve)^{\frac{1}{p-\ve}}}<\infty,
\end{equation}
then it is also bounded in the generalized grand Morrey space
\begin{equation}\label{equ:metatheorem}
\|Uf\|_{L^{q),\lb)}_{\psi,A_2}(X)} \leqslant C\|f\|_{L^{p),\lb)}_{\varphi,A_1}(X)}
\end{equation}
with
\[
 C=\frac{C_0}{\varphi(\sigma)^\frac{1}{p-\sigma}} \sup_{0<\varepsilon< \sigma} C_{p-\ve,\lambda-A_1(\ve),q-\ve,\lambda-A_2(\ve)},
\]
where $C_0$ may depend on $\gamma,p,\lb,\varphi, A$ and  $d_X,$ but does not depend on $\sg$ and $f$.
\end{lemma}

\begin{proof}
By \eqref{equ:dominant},  we have
\begin{equation}\label{equ:bound_operator_grandgrandMorrey}
\|Uf\|_{L^{q),\lambda)}_{\psi,A_2}(X)} \leqslant  \frac{C}{\psi(\sigma)^\frac{1}{q-\sigma} } \Phi^{q,\lb}_{\psi,A_2}(Uf,\sg).
\end{equation}
The estimation of $\Phi^{q,\lb}_{\psi,A_2}(Uf,\sg)$ by $\|f\|_{L^{p),\lb)}_{\varphi,A_1}(X)}$ is direct:
\begin{equation}\label{equ:meta_dominant}
\begin{split}
\Phi^{q,\lb}_{\psi,A_2}(Uf,\sg)&= \sup_{0<\varepsilon \leqslant \sigma} \psi(\varepsilon)^\frac{1}{q-\varepsilon} \|Uf\|_{L^{q-\varepsilon,\lambda-A_2(\varepsilon)}(X)}\\
&\leqslant C \sup_{0<\varepsilon \leqslant \sigma} \varphi(\varepsilon)^\frac{1}{p-\varepsilon} \cdot C_{p-\varepsilon,\lambda-A_1(\varepsilon),q-\lambda, A_2(\ve)} \cdot \|f\|_{L^{p-\varepsilon,\lambda-A(\varepsilon)}(X)}\\
& \leqslant C \sup_{0<\varepsilon \leqslant \sigma}  C_{p-\varepsilon,\lambda-A_1(\varepsilon),q-\lambda, A_2(\ve)} \cdot \|f\|_{L^{p),\lambda)}_{\varphi,A_1}(X)}
\end{split}
\end{equation}
where the first inequality comes from assumption \eqref{eq:finite_ratio}. 
\end{proof}

\begin{remark}\label{rem:up_to_constant}
The estimations in Lemmata  \ref{lem:dominance},  \ref{lem:dominance1} and \ref{main} are still true, up to a different constant,  when  we work in the modified Morrey space defined in \eqref{eq:morrey_new_norm}.
\end{remark}

\subsection{Hardy-Littlewood maximal operator}
As an application of the reduction lemma, we obtain the boundedness of the Hardy-Littlewood maximal operator
\[
Mf(x)=\sup_{\stackrel{x \in X}{0<r<d_X}} \frac{1}{\mu B(x,r)}\int_{B(x,r)} |f(y)|\dif \mu(y)
\]
in generalized grand Morrey spaces. The following proposition was shown in A. Meskhi \cite{meskhi}
\begin{proposition}\label{prop:2.2}
Let $1<p<\infty$ and $0\leqslant \lambda <1$. Then 
\[
\| Mf \|_{L^{p,\lambda}(X,\mu)}\leqslant \left((C_d)^\frac{\lambda}{p} c_0 (p^\prime)^\frac{1}{p}+1\right)\| f \|_{L^{p,\lambda}(X,\mu)}
\]
holds, where the positive constant $C_d$ arises in the doubling condition for $\mu$ and $c_0$ arises from covering lemmas. 
\end{proposition}

\begin{theorem}\label{theo:3.8}
Let $1<p<\infty$ and $0\leqslant \lambda <1$. Then the Hardy-Littlewood maximal operator is bounded from $L^{p),\lambda)}_{\varphi,A}(X, \mu)$ to $L^{p),\lambda)}_{\psi,A}(X,\mu)$ if exists small $\sg$ such that
$
\sup_{0<\ve<\sigma} \psi(\ve)^{\frac{1}{q-\ve}}/\varphi(\ve)^{\frac{1}{p-\ve}}<\infty.
$
\end{theorem}
\begin{proof}
The result follows from Lemma \ref{main} and by noticing that, from Proposition \ref{prop:2.2} and the definition of generalized grand Morrey space (see Definition \ref{def:ggms}) we have that
\[
(C_d)^\frac{\lambda-A(\varepsilon)}{p-\varepsilon}c_0 ((p-\varepsilon)^\prime)^\frac{1}{p-\varepsilon} <\infty
\]
for all $0<\varepsilon<s_{\max}$, since $\lambda-A(\varepsilon)\geqslant 0$.
\end{proof}

\subsection{Calder\'on-Zygmund singular operators}
We follow \cite{meskhi} in this section, in particular, making use of the following definition of the  Calder\'on-Zygmund singular operators.
Namely, the  Calder\'on-Zygmund  operator is defined as the integral operator
\[
Tf(x)=\mathrm{p.v.} \int_X K(x,y)f(y)\;\mathrm{d}\mu(y)
\]
with the kernel  $K:X \times X \backslash \{(x,x): x \in \Omega\} \to \mathbb R$ being a measurable function satisfying the conditions:
\[
|K(x,y)|\leqslant \frac{C}{\mu B(x,d(x,y))}, \quad x,y \in X, \quad x\neq y;
\]
\[|K(x_1,y)-K(x_2,y)|+|K(y,x_1)-K(y, x_2)| \leqslant Cw\left(\frac{d(x_2,x_1)}{d(x_2,y)}\right) \frac{1}{\mu B(x_2,d(x_2,y))}\]
for all $x_1$, $x_2$ and $y$ with $d(x_2,y)\geqslant Cd(x,x_2)$, where $w$ is a positive non-decreasing function on $(0,\infty)$ which satisfies the $\Delta_2$ condition $w(2t)\leqslant c w(t)$ ($t>0$) and the Dini condition $\int_0^1 w(t)/t \;\mathrm{d}t<\infty$. 
We also assume that  $Tf$ exists almost everywhere on $X$  in the principal value sense for all $f \in L^{2}(X)$ and that $T$ is bounded in $L^{2}(X).$ \\
The boundedness of such Calder\'on-Zygmund operators in Morrey spaces is valid, as can be seen in the following Proposition, proved in \cite{meskhi}.
\begin{proposition}\label{prop:boundedness_ron}Let 
 $1<p<\infty$ and $0\leqslant \lambda <1$. Then
\[
\|Tf\|_{L^{p,\lambda}(X,\mu)} \leqslant C_{p,\lambda} \|f\|_{L^{p,\lambda}(X,\mu)}
\]
where 
\begin{equation}\label{equ:constant_Morrey_calderon}
C_{p,\lambda}\leqslant  c\left\{
  \begin{array}{ll}
     \frac{p}{p-1}+\frac{p}{2-p} +\frac{p-\lambda+1}{1-\lambda} & \mbox{if } 1<p<2, \\
    p+\frac{p}{p-2} +\frac{p-\lambda+1}{1-\lambda} & \mbox{if } p>2,\\
  \end{array}
\right.
\end{equation}
with $c$ not depending on $p$ and $\lambda$.
\end{proposition}
\begin{theorem}Let   $1<p<\infty$, $\theta >0$, $\alpha \geqslant 0$ and $0<\lambda<1$. Then the Calder\'on-Zygmund operator $T$ is bounded in generalized grand Morrey spaces $L^{p),\lambda)}_{\varphi, A}(X,\mu)$.
\end{theorem}
\begin{proof}
  Keeping in mind that by Lemma \ref{main} we are interested only in small values of $\varepsilon$, from  \eqref{equ:constant_Morrey_calderon}, we deduce that
\[
C_{p-\varepsilon,\lambda-A(\varepsilon)}\leqslant
c \left\{
  \begin{array}{ll}
     \frac{p}{p-\varepsilon-1}+\frac{p-\varepsilon}{2-p+\varepsilon} +\frac{p-\varepsilon-\lambda+A(\varepsilon)+1}{1-\lambda+A(\varepsilon)} & \mbox{if } p\leqslant 2 \ \mbox{and} \ \ 0<\varepsilon<p-1; \\
\\
    p-\varepsilon+\frac{p-\varepsilon}{p-\varepsilon-2} +\frac{p-\varepsilon-\lambda+A(\varepsilon)+1}{1-\lambda+A(\varepsilon)} & \mbox{if } p>2 \ \ \mbox{and}\ \ 0<\varepsilon<p-2.\\
  \end{array}
\right.
\]
and we are done, since it is possible to choose a small $\sigma$  such that \eqref{reduction_condition} is valid.
\end{proof}

\section{Riesz type potential operators in generalized grand Morrey spaces} In this section we will assume that the triplet $(X,d,\mu)$ is an  SHT  and the measure $\mu$ is upper $\gamma$-Ahfors regular.

\subsection{Riesz potential operator} 
By a \textit{Riesz type potential operator}, we mean an operator of the type
\begin{equation}\label{eq:riesz}
I^\alpha f(x):=\int_X \frac{f(y)}{d(x,y)^{\gamma-\alpha}} \dif \mu(y)
\end{equation}
where $0<\alpha<\gamma$.

The following proposition was shown in A. Meskhi \cite{meskhi} in the case of Morrey spaces $\mathscr L^{q,\lambda}(X,\mu)$ as defined in \eqref{eq:morrey_new_norm}.

\begin{proposition}\label{lem:meskhi}
Let $1<p<\infty$, $0<\alpha < \frac{(1-\lambda)\gamma}{p}$, $\frac{1}{p}-\frac{1}{q}=\frac{\alpha}{(1-\lambda)\gamma}$, where $0\leqslant \lambda <1$. Then the inequality
\begin{equation}\label{equ:meskhi}
\|I^{\alpha} f\|_{\mathscr L^{q,\lambda}(X,\mu)} \leqslant \overline c (p,\alpha,\lambda,\gamma) \|f\|_{\mathscr L^{p,\lambda}(X,\mu)}
\end{equation}
holds, where the positive constant $\overline  c (p,\alpha,\lambda,\gamma)$  is given by
\[
\overline  c (p,\alpha,\lambda,\gamma)=c \frac{(1-\lambda)\gamma}{\alpha[(1-\lambda)\gamma-\alpha p]} [(p^\prime)^{1/q}+1]
\]
and the positive constant $c$ does not depend on $p$ and $\alpha$.
\end{proposition}

Using the norm \eqref{eq:morrey_new_norm}, we define the corresponding generalized grand Morrey space, namely
\begin{equation}\label{eq:grandgrandmorrey_mew}
\|f\|_{\mathscr L^{p),\lambda)}_{\varphi,A}(X,\mu)}:=\sup_{0<\ve<s_{\max}} \varphi(\ve)^\frac{1}{p-\ve} \|f\|_{\mathscr L^{p-\ve,\lambda-A(\ve)}(X,\mu)}
\end{equation}
where $s_{\max}=\min\{p-1,a\}$ with $a=\sup\{x>0:A(x)\leqslant \lambda \}$. If $\varphi(\ve):= \ve^\theta$, when $\theta$ is a positive number, we denote $\|f\|_{\mathscr L^{p),\lambda)}_{\varphi,A}(X,\mu)}=:\|f\|_{\mathscr L^{p),\lambda)}_{\theta,A}(X,\mu)}$.

By the Hardy-Littlewood-Sobolev inequality we know that the boundedness of the Riesz potential operator in Lebesgue spaces is valid when the exponents are different and related to each other. In the case of generalized grand Morrey spaces we will have a similar result, but now not only the exponents will be different, also the indices $\theta_1$ and $\theta_2$ will be different.

Before stating and proving the main result in this subsection (Theorem \ref{theo:boundedness_of_riesz}), we introduce some auxiliary functions those will be used afterwards.

\begin{definition}[auxiliary functions]\label{eq:auxiliary_functions}
On an interval $(0,\delta ]$, $\delta $ is small,  we define the following
functions:
\[
\bar{\phi}(x):=p+\frac{\gamma(x-q)(1-\lambda +A_{2}(x))}{\gamma(1-\lambda
+A_{2}(x))-\alpha (x-q)},\; \tilde{\phi}(x):=q-\frac{\gamma(p-x)(1-\lambda
+A_{1}(x))}{\gamma(1-\lambda +A_{1}(x))-\alpha (p-x)}
\]

\[
\bar{A}(x)=1-\frac{\alpha(x-q)}{\gamma(1-\lambda +A_{2}(x))}, \; \tilde{A}(x)=\frac{%
1-\lambda +A_{1}(\eta )}{\gamma(1-\lambda +A_{1}(\eta ))-(p-\eta )\alpha }
\]

\[
\phi (x):= \bar{\phi}(x)^{\bar{A}(x)},\quad 
\Phi (x):= \tilde{\phi}(x)^{\tilde{A}(x)}
\]

\[
\psi (\varepsilon )=\phi (\varepsilon ^{\theta _{1}}), \quad 
\text{ }\Psi (\varepsilon )=\Phi (\varepsilon ^{\theta _{1}}), 
\]
for $\theta_{1}>0$.
\end{definition}

\begin{theorem}\label{theo:boundedness_of_riesz}
Let $1<p<\infty,$ $0<\alpha <((1-\lambda)\gamma)/p,$ $0<\lambda <1,$  $1/p-1/q=\alpha/((1-\lambda)\gamma).$ Suppose that $\theta _{1}>0$ and that $\theta _{2}\geqslant\theta
_{1}[1+\alpha q/((1-\lambda)\gamma )].$ Let $A_{1}$ and $A_{2}$ be continuous
non-negative functions on $(0,p-1]$ and $(0,q-1]$ respectively satisfying
the conditions:
\begin{itemize}
\item[(i)] $A_{2}\in C^{1}((0,\delta ])$ for some positive $\delta >0;$
\item[(ii)] $\lim_{x\rightarrow 0+}A_{2}(x)=0;$
\item[(iii)] $0\leqslant B:=\lim_{x\rightarrow 0+} \frac{d A_{2}}{dx}(x)<\frac{%
(1-\lambda )^{2}}{\alpha q^{2}};$ 
\item[(iv)] $A_{1}(\eta )=A_{2}(\bar{\phi}^{-1}(\eta )),$ where $\bar{\phi}%
^{-1}$ is the inverse of $\bar{\phi}$ on $(0,\delta ]$ for some $\delta >0.$
\end{itemize}

Then the Riesz potential operator $I^{\alpha }$ is $\left(\mathscr L_{\theta _{1},A_{1}}^{p),\lambda)}(X,\mu) - \mathscr L_{\theta _{2},A_{2}}^{q),\lambda)}(X,\mu)\right)$-bounded.
\end{theorem}

\begin{proof}
We note that it is enough to prove the theorem for $\theta _{2}=\theta _{1}(1+\frac{%
\alpha q}{1-\lambda })$ because $\varepsilon ^{\theta _{2}}\leqslant \varepsilon
^{\theta _{1}(1+\frac{\alpha q}{1-\lambda })}$ for $\theta _{2}>\theta
_{1}[1+(\alpha q)/(1-\lambda)]$ and small $\ve$. We also note that, by L'Hospital rule,  $\bar{\phi}(x)\sim x$ as $\ x\rightarrow 0+$ since $B<(1-\lambda )^2/(\alpha q^2)$. Moreover, $\bar{\phi}$ is invertible near $0$, since $\frac{d \bar{\phi}}{dx}(x)>0.$ Under the conditions of the Theorem \ref{theo:boundedness_of_riesz} the function $A_{1}$ is continuous on $(0,\delta ]$ and $\lim_{x\rightarrow0+}A_{1} (x)=0.$ With all of the previous remarks taken into account, it is enough to prove the boundedness of $I^\alpha$ from $\mathscr L_{\theta _{1},A_{1}}^{p),\lambda)}(X,\mu)$
to $\mathscr L_{\psi, A_{2}}^{q),\lambda)}(X,\mu)$ since $\phi(x) \sim  x^{1+\frac{\alpha q}{%
(1-\lambda)\gamma }}$, and consequently, $\psi (x)=\phi (x^{\theta _{1}})\sim
x^{\theta _{1}\left(1+\frac{\alpha q}{(1-\lambda)\gamma }\right)}$ as $x\rightarrow 0. $
We will now proceed in a similar fashion as in the proof of Lemma \ref{lem:dominance}.

The case $\sigma<\varepsilon\leqslant s_{\max}$, where $s_{\max}$ is  from \eqref{eq:grandgrandmorrey_mew}. Letting
\[
I:=\psi^{\frac{1}{q-\varepsilon}}(\varepsilon)\left(\frac{1}{\mu B(x,r)^{\lambda-A_2(\varepsilon)}} \int_{B(x,r)} |I^\alpha f(y)|^{q-\varepsilon} \dif \mu(y) \right)^\frac{1}{q-\varepsilon}
\]
we have
\[
\begin{split}
I&\leqslant \psi^{\frac{1}{q-\varepsilon}}(\varepsilon) \mu B(x,r)^\frac{A_2(\varepsilon)+1-\lambda}{q-\varepsilon}\left( \fint_{B(x,r)} |I^\alpha f(y)|^{q-\sigma} \dif \mu(y) \right)^\frac{1}{q-\sigma}\\
&\leqslant C \psi^{\frac{1}{q-\varepsilon}}(\varepsilon) \mu B(x,r)^\frac{A_2(\sigma)+1-\lambda}{q-\sigma}\left( \fint_{B(x,r)} |I^\alpha f(y)|^{q-\sigma} \dif \mu(y) \right)^\frac{1}{q-\sigma}\\
&\leqslant C \left (\sup_{\sigma\leqslant \varepsilon\leqslant s_{\max}}\psi^{\frac{1}{q-\varepsilon}}(\varepsilon)\right ) \psi^\frac{1}{\sigma-q}(\sigma) \times \\
& \hspace{2cm} \sup_{0<\varepsilon\leqslant \sigma} \sup_{\stackrel{x\in X}{r>0}} \left(\frac{\psi(\varepsilon)}{\mu B(x,r)^{\lambda-A_2(\varepsilon)}} \int_{B(x,r)} |I^\alpha f(y)|^{q-\varepsilon} \dif \mu(y) \right)^\frac{1}{q-\varepsilon}
\end{split}
\]
where the first inequality comes from H\"older's inequality and the second one is due to the fact that $A_2$ is bounded on $[\sigma,q-1)$ and $x\mapsto (1-\lambda)/(q-x)$ is an increasing function. Hence, it is enough to consider the case $0<\varepsilon\leqslant \sigma $.\par
The case $0<\varepsilon\leqslant \sigma$. Let $\eta$ and $\varepsilon$ be chosen so that 
\begin{equation}\label{eq:etavar}
\frac{1}{p-\eta}-\frac{1}{q-\varepsilon}=\frac{\alpha}{(1-\lambda+A_2(\varepsilon))\gamma}.
\end{equation}
Obviously we have that $\varepsilon \to 0$ if and only if  $\eta \to 0$ and solving $\eta$ with  respect to $\varepsilon$ in \eqref{eq:etavar} we obtain  
\[
\eta=p-\frac{\gamma(q-\varepsilon)(1-\lambda+A_2(\varepsilon))}{\gamma(1-\lambda+A_2(\varepsilon))-\alpha(\varepsilon-q)}= \bar \phi(\varepsilon).
\]
Letting 
\[
J:=\psi^{\frac{1}{q-\varepsilon}}(\varepsilon)\left(\frac{1}{\mu B(x,r)^{\lambda-A_2(\varepsilon)}} \int_{B(x,r)} |I^\alpha f(y)|^{q-\varepsilon} \dif \mu(y) \right)^\frac{1}{q-\varepsilon}
\]
we have
\[
\begin{split}
J&\leqslant C \frac{(1-\lambda+A_2(\varepsilon))\gamma}{\alpha[(1-\lambda+A_2(\varepsilon))\gamma-\alpha(p-\eta)]}[(p-\eta)^\prime]^\frac{1}{q-\varepsilon}\psi^\frac{1}{q-\varepsilon}(\varepsilon) \times\\
&\hspace{2cm}\sup_{\stackrel{x\in X}{r>0}}\left( \frac{1}{\mu B(x,r)^{\lambda-A_2(\varepsilon)}} \int_{B(x,r)} |f(y)|^{p-\eta} \dif \mu(y) \right)^\frac{1}{p-\eta} \\
&\leqslant C \frac{(1-\lambda+A_2(\varepsilon))\gamma}{\alpha[(1-\lambda+A_2(\varepsilon))\gamma-\alpha(p-\eta)]}\eta^\frac{\theta_1}{\eta-p}\psi^\frac{1}{q-\varepsilon}(\varepsilon) \times\\
&\hspace{2cm}\sup_{\stackrel{x\in X}{r>0}}\left( \frac{\eta^{\theta_1}}{\mu B(x,r)^{\lambda-A_2(\varepsilon)}} \int_{B(x,r)} |f(x)|^{p-\eta} \dif \mu(x) \right)^\frac{1}{p-\eta} \\
&\leqslant C \|f\|_{\mathscr L^{p),\lambda)}_{\theta_1,A_1}(X,\mu)}
\end{split}
\]
where the first inequality is due to Proposition \ref{lem:meskhi} and the last one is due to the fact that $\eta=\bar \phi(\varepsilon)$. Since the constant in the last inequality is uniformly bounded with respect  to  $\varepsilon$  we obtain the desired boundedness of the Riesz potential operator.
\end{proof}

\begin{corollary}
Let $1<p<\infty,$ $0<\alpha <((1-\lambda)\gamma)/p,$ $0<\lambda <1,$  $1/p-1/q=\alpha/((1-\lambda)\gamma).$ Suppose that $\theta _{1}>0$ and that $\theta _{2}\geqslant\theta
_{1}(1+\alpha q/(1-\lambda ))$.
 Let $A_2(x)=\alpha x$, where $\alpha$ is a non-negative constant satisfying the condition $\alpha<(1-\lambda)^2/(\alpha q^2)$. Let $A_1(x)=\alpha \bar \phi^{-1}(x)$. Then $I^\alpha$ is  $\left(\mathscr L_{\theta _{1},A_{1}}^{p),\lambda)}(X,\mu) - \mathscr L_{\theta _{2},A_{2}}^{q),\lambda)}(X,\mu)\right)$-bounded.
\end{corollary}

It is also possible to prove a similar result of Theorem \ref{theo:boundedness_of_riesz}, but now requiring conditions on the function $A_1$, namely we have

\begin{theorem}
Let\textbf{\ }$1<p<\infty,$ $0<\alpha <1,$ $0<\lambda <1-\alpha p,$  $\frac{1}{p}-\frac{1}{q}=\frac{\alpha }{1-\lambda }.$ Suppose that $\theta _{1}>0$ and that $\theta _{2}>\theta
_{1}(1+\frac{\alpha q}{1-\lambda }).$ Let $A_{1}$ and $A_{2}$ be continuous
non-negative functions on $(0,p-1]$ and $(0,q-1]$ respectively satisfying
the conditions:
\begin{itemize}
\item[(i)] $A_{1}\in C^{1}((0,\delta ])$ for some positive $\delta >0;$
\item[(ii)] $\lim_{x\rightarrow 0+}A_{1}(x)=0;$
\item[(iii)] $ B_1:=\lim_{x\rightarrow 0+} \frac{d A_{1}}{dx}(x) \geqslant 0;$
\item[(iv)] $A_{2}(x )=A_{1}(\tilde{\phi}^{-1}(x ))$  on $(0,\delta ]$ for some $\delta >0.$
\end{itemize}

Then the Riesz potential operator $I^{\alpha }$ is $\left(\mathscr L_{\theta _{1},A_{1}}^{p),\lambda)}(X,\mu) - \mathscr L_{\theta _{2},A_{2}}^{q),\lambda)}(X,\mu)\right)$-bounded.

\end{theorem}
\begin{proof}
The proof is similar to that of Theorem \ref{theo:boundedness_of_riesz}. In this case it is enough to prove that  $I^\alpha$ is  $\left(\mathscr L_{\Psi,A_{1}}^{p),\lambda)}(X,\mu) - \mathscr L_{\theta _{2},A_{2}}^{q),\lambda)}(X,\mu)\right)$-bounded, where $\theta_2=1+\frac{\alpha q}{1-\lambda}=\frac{1-\lambda}{1-\lambda-\alpha p}$.
\end{proof}

\subsection{Riesz potential operator defined via measure}By a \textit{Riesz type potential operator defined via measure}, we mean an operator of the type
\begin{equation}\label{eq:riesz_measure}
I^\alpha_\mu f(x):=\int_X \frac{f(y)}{\mu(x,d(x,y))^{1-\alpha}} \dif \mu(y)
\end{equation}
where $0<\alpha<1$. 

The following  proposition was proved in A. Meskhi \cite{meskhi}.

\begin{proposition}\label{lem:meskhi_measure}
Let $1<p<\infty$, $0<\alpha < \frac{1-\lambda}{p}$, $\frac{1}{p}-\frac{1}{q}=\frac{\alpha}{1-\lambda}$, where $0\leqslant \lambda <1$. Then the inequality
\begin{equation}\label{equ:meskhi}
\|I^{\alpha}_\mu f\|_{ L^{q,\lambda}(X,\mu)} \leqslant c(p,\alpha,\lambda) \|f\|_{L^{p,\lambda}(X,\mu)}
\end{equation}
holds, where
\[
c(p,\alpha,\lambda)=b_0 \left (C_\alpha+\frac{p}{1-\lambda-\alpha p} \right) [(p^\prime)^{1/q}+1]
\]
and the positive constant $b_0$ does not depend on $p$ and $\alpha$.
\end{proposition}

It is also possible to obtain the boundedness of the operator $I^{\alpha}_\mu$ in the framework of generalized grand Morrey spaces, namely the following statement holds:

\begin{theorem}\label{theo:boundedness_of_riesz_measure}
Let $1<p<\infty,$ $0<\alpha <(1-\lambda)/p,$ $0<\lambda <1,$  $1/p-1/q=\alpha/(1-\lambda).$ Suppose that $\theta _{1}>0$ and that $\theta _{2}\geqslant\theta
_{1}(1+\alpha q/(1-\lambda )).$ Let $A_{1}$ and $A_{2}$ be continuous
non-negative functions on $(0,p-1]$ and $(0,q-1]$ respectively satisfying
the conditions:
\begin{itemize}
\item[(i)] $A_{2}\in C^{1}((0,\delta ])$ for some positive $\delta >0;$
\item[(ii)] $\lim_{x\rightarrow 0+}A_{2}(x)=0;$
\item[(iii)] $0\leqslant B:=\lim_{x\rightarrow 0+} \frac{d A_{2}}{dx}(x)<\frac{%
(1-\lambda )^{2}}{\alpha q^{2}};$ 
\item[(iv)] $A_{1}(\eta )=A_{2}(\bar{\phi}^{-1}(\eta )),$ where $\bar{\phi}%
^{-1}$ is the inverse of $\bar{\phi}$ on $(0,\delta ]$ for some $\delta >0.$
\end{itemize}

Then the Riesz potential operator $I^{\alpha }$ is $\left( L_{\theta _{1},A_{1}}^{p),\lambda)}(X,\mu) - L_{\theta _{2},A_{2}}^{q),\lambda)}(X,\mu)\right)$-bounded.
\end{theorem}

\begin{proof}
The proof follows, mutatis mutandis, the proof of Theorem \ref{theo:boundedness_of_riesz}. Namely, we need to use Proposition \ref{lem:meskhi_measure} and the auxiliary functions from Definition \ref{eq:auxiliary_functions} should be used with $\gamma=1$.
\end{proof}

\begin{remark}
If $\mu$ is upper Ahlfors regular, then by using the pointwise estimate:
$$ I^{\alpha}f(x) \leq c_{\alpha} Mf(x), \;\; f\geq 0, $$
and Theorem \ref{theo:3.8} we have also the same boundedness for $I^{\alpha}$ as in Theorem \ref{theo:3.8} for $M$. 
\end{remark}

\section{Potentials on nonhomogeneous spaces}
In this section we will deal with potential operators in the framework of  nonhomogeneous spaces. Namely, 
let $(X, d, \mu)$ be a topological space with a complete measure $\mu$ such that the space of compactly supported functions are dense in $L^1(X, \mu)$ and $d$ is a quasimetric satisfying the standard conditions, see subsection \ref{preliminaries}. As before we will assume that $d_X\equiv \mathrm{diam} (X)<\infty$. In this section,  we do \textsc{not} assume that $\mu$ is doubling! 

Let 
\[
(K_\alpha f)(x)=\int_X \frac{f(y)}{d(x,y)^{1-\alpha}}\dif \mu(y), 
\]
where $0<\alpha<1$.

We need the following {\it modified} maximal operator on $X$
\[
(\widetilde M f)(x)=\sup_{r>0} \frac{1}{\mu B(x_0,N_0r)}\int_{B(x,r)} |f(y)| \dif \mu(y),
\]
where $N_0=C_t(1+2C_s)$ and the constants $C_s$ and $C_t$ are from the definition of quasimetric $d$. Let $b$ be a constant. We will use the symbol $bB$ for a ball $B(x,br)$, where $B\equiv B(x,r)$.

\begin{lemma}\label{lemma1}
Let $1<p<\infty$. Then the following inequality holds for all $f \in L^p(X, \mu)$
\begin{equation}\label{equacao1}
\|\widetilde M f\|_{L^p(X, \mu)}\leqslant 2 (p^\prime)^\frac{1}{p}\| f\|_{L^p(X, \mu)}
\end{equation}
\end{lemma}

\begin{proof}
The operator $\widetilde M$ is of  weak type $(1,1)$  with constant 1, i.e. the inequality
\[
\mu\left(\left\{x \in X: (\widetilde Mf )(x)>\lambda\right\} \right)\leqslant \frac{1}{\lambda} \int_X |f(x)|\dif \mu(x) 
\]
holds, see \cite[p. 368]{EdKoMe}. Since $\widetilde M$ is of strong type  $(\infty,\infty)$ with constant 1, i.e. $\|\widetilde M f\|_{L^\infty}\leqslant \|f\|_{L^\infty}$, we conclude that the inequality \eqref{equacao1} holds with constant $2(p^\prime)^\frac{1}{p}$ (see \cite[p. 29]{duoandikoetxea}).
\end{proof}

\subsection{Modified Morrey space}
We will define a modified Morrey space
\begin{definition}[Modified Morrey space]
Let $1<p<\infty$ and let $0\leqslant \lambda <1$. Suppose that $a$ is a positive constant. We denote by $L^{p,\lambda}(X,\mu)_a$ the modified Morrey space defined by the norm
\[
\|f\|_{L^{p,\lambda}(X,\mu)_a}=\sup_{x \in X, r>0} \left (\frac{1}{\mu B(x,ar)^\lambda} \int_{B(x,r)} |f(y)| \dif \mu(y) \right)
\]
\end{definition}

For the next statement we refer to \cite{KMCVEE}, but we give the proof for completeness, because we will need the constant in the inequality.

\begin{lemma}\label{lemma2}
Let $1<p<\infty$ and let $0\leqslant \lambda <1$. Then the following inequality
\begin{equation}\label{equacaolemma2}
\|\widetilde M f\|_{L^{p,\lambda}(X, \mu)_{N_0 \overline{a}}}\leqslant \left[1+2 (p^\prime)^\frac{1}{p}\right]\| f\|_{L^{p,\lambda}(X, \mu)_{N_0}}
\end{equation}
holds, where $N_0$ and $\overline a$ are positive constants defined by $N_0=C_t(1+2C_s)$, $\overline a=C_t(C_t(C_s+1)+1)$.
\end{lemma}

\begin{proof}Let $r$ be a small positive number and represent $f$ as follows $f=f_1+f_2$, where $f_1=f\cdot\chi_{B(x,\overline a r)}$, $f_2=f-f_1$ and $\overline a$ is the constant defined above.

We have

\begin{multline*}
\left[\frac{1}{\mu B(x,N_0\overline ar)^\lambda} \int_{B(x,r)} (\widetilde M f)^p(y) \dif \mu(y)\right]^\frac{1}{p} \leqslant \\ \left[\frac{1}{\mu B(x,N_0\overline ar)^\lambda} \int_{B(x,r)} (\widetilde M f_1)^p(y) \dif \mu(y)\right]^\frac{1}{p} + \\
\left[\frac{1}{\mu B(x,N_0\overline ar)^\lambda} \int_{B(x,r)} (\widetilde M f_2)^p(y) \dif \mu(y)\right]^\frac{1}{p}=:\\J_1(x,r)+J_2(x,r).
\end{multline*}
By applying Lemma \ref{lemma1}, we have that
\[
\begin{split}
J_1(x,r) &\leqslant \frac{1}{\mu B(x,N_0\overline ar)^\frac{\lambda}{p}} \left(\int_{B(x,r)} (\widetilde M f_1)^p(y) \dif \mu(y)\right)^\frac{1}{p} \\
&\leqslant 2(p^\prime)^\frac{1}{p}[\mu B(x,N_0\overline ar)]^{-\frac{\lambda}{p}} \left(\int_{B(x,\overline a r)} ( f)^p(y) \dif \mu(y)\right)^\frac{1}{p}\\
&\leqslant 2(p^\prime)^\frac{1}{p} \|f\|_{L^{p,\lambda}(X,\mu)_{N_0}}.
\end{split}
\]
Observe now that (see also \cite[p. 929]{KMCVEE}) if $y \in B(x,r)$, then $B(x,r)\subset B(y,C_t(C_s+1)r)\subset B(x,\overline a r)$. Hence, for $y \in B(x,r)$
\[
(\widetilde M f_2)(y)\leqslant \sup_{B\supset B(x,r)} \frac{1}{\mu(N_0B)}\int_B |f(z)|\dif \mu(z).
\]
Consequently
\[
\begin{split}
J_2(x,r)&\leqslant \left[ \mu B(x,N_0\overline ar) \right]^{-\frac{\lambda}{p}}\sup_{B\supset B(x,r)}\left[\frac{1}{\mu (N_0B)} \int_B |f(y)|\dif \mu(y) \right] \left(\mu B(x,r) \right)^\frac{1}{p}\\
&\leqslant \left(\mu B(x,N_0r) \right)^\frac{1-\lambda}{p}\sup_{B\supset B(x,r)}\left[\frac{1}{\mu (N_0B)} \int_B |f(y)|^p \dif \mu(y) \right]^\frac{1}{p}\\
&\leqslant \sup_B \left[\frac{1}{\mu (N_0B)^\frac{\lambda}{p}} \int_B |f(y)|^p \dif \mu(y) \right]^\frac{1}{p}\\
&=\|f\|_{L^{p,\lambda}(X,\mu)_{N_0}}
\end{split}
\]
Therefore we obtain \eqref{equacaolemma2}.
\end{proof}

For the next statement we refer to \cite{kok_90} for Euclidean spaces and \cite[p. 367]{EdKoMe} for nonhomogeneous spaces.

\begin{theorema}
Let $1<p<\infty$, $0<\alpha<1/p$, $q= p/(1-\alpha p)$. Then $K_\alpha$ is bounded from $L^p(X,\mu)$ to $L^q(X,\mu)$ if and only if there is a positive constant $b$ such that
\begin{equation}\label{equacao3}
\mu B(x,r)\leqslant br
\end{equation}
for all $x \in X$ and $r>0$, i.e., the measure is upper 1-Ahlfors regular.
\end{theorema}

\begin{remark}
 Theorem A is proved for $\mu X =\infty$ but it is also true for $\mu X<\infty$ (observe that $d_X<\infty$ implies $\mu X <\infty$ because $\mu B <\infty$ for all balls!)
\end{remark}

\begin{lemma}\label{lemma3}
Let the measure $\mu$ be upper 1-Ahlfors regular, $1<p<\infty$, $0<\alpha<\frac{1-\lambda}{p}$, $0\leqslant \lambda<1$.  We set $q=\frac{p(1-\lambda)}{1-\lambda-\alpha p}$. Then the following inequality holds:
\[
\|K_\alpha f\|_{L^{q,\lambda}(X,\mu)_{N_0\overline a}} \leqslant C_{b,N_0,p,\lambda,\alpha} \| f\|_{L^{p,\lambda}(X,\mu)_{N_0}},
\]
where
\[
C_{b,N_0,p,\lambda,\alpha}=4\left[ 1+2(p^\prime)^\frac{1}{p}\right]^\frac{p}{q}\left[ \frac{bN_0}{\alpha}+\frac{b^{\frac{1}{p_1}-\frac{\lambda}{p}}N_0^\frac{\lambda}{p}p}{1-\lambda-\alpha p} \right].
\]
with $b$ from \eqref{equacao3}.
\end{lemma}

\begin{proof}
First we prove the Hedberg's type inequality
\[
|K_\alpha f(x)|\leqslant A_{b,N_0,p,\lambda,\alpha} \widetilde Mf(x)^{1-\frac{p\alpha}{1-\lambda}}\|f\|^{\frac{\alpha p}{1-\lambda}}_{L^{p,\lambda}(X,\mu)_{N_0}},
\]
where
\[
A_{b,N_0,p,\lambda,\alpha}=4\left[ \frac{bN_0}{\alpha}+\frac{b^{\frac{1}{p_1}-\frac{\lambda}{p}}N_0^\frac{\lambda}{p}p}{1-\lambda-\alpha p} \right].
\]

Observe that the inequality $d(x,y)^{\alpha-1}\leqslant 2^{2-\alpha}\int_{d(x,y)}^{2d(x,y)} t^{\alpha-2}\dif t$ holds, where $0<d(x,y)<d_X$. Hence,
\[
\begin{split}
|K_\alpha f(x)|&\leqslant 4 \int_X |f(y)| \left ( \int_{d(x,y)}^{2d(x,y)} t^{\alpha-2}\dif t \right) \dif \mu(y)\\
&=4 \int_0^{2d_X} t^{\alpha-2} \left(  \int_{t/2 <d(x,y)<t} |f(y)|\dif \mu(y) \right)\dif t =4\left (\int_0^\ve +\int_\ve^{2d_X}  \right) \ldots \\
&=4(S_1+S_2).
\end{split}
\]
By using condition \eqref{equacao3} we find that
\[
\begin{split}
S_2&\leqslant \int_0^\ve t^{\alpha-1}  \left( \frac{1}{t} \int_{B(x,t)} |f(y)|\dif \mu(y) \right) \dif t \\
&\leqslant bN_0 \left( \int_0^\ve t^{\alpha-1} \dif t \right) \widetilde Mf(x)\\
&=\frac{bN_0}{\alpha} \ve^\alpha \widetilde M f(x),
\end{split}
\]
where $b$ is the constant  from \eqref{equacao3}.

Further, H\"older's inequality and condition \eqref{equacao3} yields
 \[
\begin{split}
\frac{1}{t} \int_{B(x,t)} |f(y)|\dif \mu(y)&\leqslant \frac{(\mu B(x,t))^\frac{1}{p^\prime}}{t}\left(\int_{B(x,t)} |f(y)|^p \dif \mu(y) \right)^\frac{1}{p}\\
&\leqslant \frac{(\mu B(x,t))^\frac{1}{p^\prime} (\mu B(x,N_0t))^\frac{\lambda}{p}}{t} \times \\ 
& \hspace{2.5cm} \times \left(\frac{1}{\mu B(x,N_0t))^\lambda}\int\limits_{B(x,t)} |f(y)|^p \dif \mu(y) \right)^\frac{1}{p}\\
&\leqslant b^{\frac{1}{p^\prime}+\frac{\lambda}{p}}N_0^\frac{\lambda}{p} t^\frac{\lambda-1}{p}\|f\|_{L^{p,\lambda}(X,\mu)_{N_0}}.
\end{split}
\]
Hence
\[
\begin{split}
(K_\alpha f)(x)&\leqslant 4\left[ \frac{bN_0}{\alpha} \ve^\alpha (\widetilde M f)(x) +  b^{\frac{1}{p^\prime}+\frac{\lambda}{p}}N_0^\frac{\lambda}{p} \left(\int_\ve^{2\ve} t^{\frac{\lambda-1}{p}+\alpha-1} \dif t\right) \|f\|_{L^{p,\lambda}(X,\mu)_{N_0}} \right]\\
&= 4\left[ \frac{bN_0}{\alpha} \ve^\alpha (\widetilde M f)(x) +  b^{\frac{1}{p^\prime}+\frac{\lambda}{p}}N_0^\frac{\lambda}{p}  \frac{\ve^{\frac{\lambda-1}{p}+\alpha}}{\frac{1-\lambda}{p}-\alpha} \|f\|_{L^{p,\lambda}(X,\mu)_{N_0}} \right]\\
&=\left[ \frac{bN_0}{\alpha} \ve^\alpha (\widetilde M f)(x) +  \frac{b^{\frac{1}{p^\prime}-\frac{\lambda}{p}}N_0^\frac{\lambda}{p}p}{1-\lambda-\alpha p}  \ve^{\frac{\lambda-1}{p}+\alpha} \|f\|_{L^{p,\lambda}(X,\mu)_{N_0}} \right].
\end{split}
\]
Let us take 
\[
\ve= \left[\frac{\|f\|_{L^{p,\lambda}(X,\mu)_{N_0}}}{(\widetilde M f)(x)} \right]^\frac{p}{1-\lambda}
\]
Then
\[
(K_\alpha f(x))\leqslant \underbrace{4 \left[ \frac{bN_0}{\alpha} +  \frac{b^{\frac{1}{p^\prime}-\frac{\lambda}{p}}N_0^\frac{\lambda}{p}p}{1-\lambda-\alpha p}  \right]}_{=:A_{b,N_0,p,\lambda,\alpha}}  \|f\|_{L^{p,\lambda}(X,\mu)_{N_0}}^\frac{\alpha-p}{1-\lambda} (\widetilde M f)^{1-\frac{p\alpha}{1-\lambda}}(x) .
\]

Finally, letting
\[
J:=\left(\frac{1}{\mu B(x,N_0\overline ar)^\lambda} \int_{B(x,r)} |K_\alpha f(y)|^q \dif \mu(y) \right)^\frac{1}{q}
\]
by Lemma \ref{lemma2} we have
\[
\begin{split}
J&\leqslant A_{b,N_0,p,\lambda,\alpha} \|f\|_{L^{p,\lambda}(X,\mu)_{N_0}}^\frac{\alpha-p}{1-\lambda} 
 \left[ \frac{1}{\mu B(x,N_0\overline ar)^\lambda} \int_{B(x,r)} (\widetilde M f(y))^q d\mu(y) \right]^\frac{1}{q}\\
&\leqslant A_{b,N_0,p,\lambda,\alpha} \|f\|_{L^{p,\lambda}(X,\mu)_{N_0}}^\frac{\alpha-p}{1-\lambda} \|\widetilde Mf\|_{L^{p,\lambda}(X,\mu)_{N_0\overline a}}^\frac{p}{q}\\
& \leqslant \left[1+2(p^\prime)^\frac{1}{p} \right]^\frac{p}{2}A_{b,N_0,p,\lambda,\sigma}  \|f\|_{L^{p,\lambda}(X,\mu)_{N_0}}. \qedhere
\end{split}
\]
\end{proof}

Let us study the boundedness of the operator $K_\alpha$ from the space 
$L^{p),\lambda)}_{\theta_1,A_1}(X, \mu)_{N_0}$ to $L^{q),\lambda)}_{\theta_2,A_2}(X, \mu)_{N_0\overline a}$, where $p,q$ and $\lambda$ satisfy the conditions of Lemma \ref{lemma3} and the space $L^{p),\lambda)}_{\theta,A}(X,\mu)_a$ is defined by the norm
\[
\|f\|_{L^{p),\lambda)}_{\theta,A}(X,\mu)_a}=\sup_{0<\ve \leqslant p-1} \sup_{ \substack{ x\in X \\ 0<r<d_X}} \left[\frac{\ve^\theta}{\mu B(x,ar)^{\lambda-A(\ve)}} \int_{B(x,r)} |f(y)|^{p-\ve} d\mu(y) \right]^\frac{1}{p-\ve}.
\]

\begin{theorem}\label{theo:boundedness_of_riesz_nonhomogeneous}
Let $1<p<\infty,$ $0<\alpha <(1-\lambda)/p,$ $0<\lambda <1,$  $1/p-1/q=\alpha/(1-\lambda).$ Suppose that $\theta _{1}>0$ and that $\theta _{2}\geqslant\theta
_{1}(1+\alpha q/(1-\lambda )).$ Let $A_{1}$ and $A_{2}$ be continuous
non-negative functions on $(0,p-1]$ and $(0,q-1]$ respectively satisfying
the conditions:
\begin{itemize}
\item[(i)] $A_{2}\in C^{1}((0,\delta ])$ for some positive $\delta >0;$
\item[(ii)] $\lim_{x\rightarrow 0+}A_{2}(x)=0;$
\item[(iii)] $0\leqslant B:=\lim_{x\rightarrow 0+} \frac{d A_{2}}{dx}(x)<\frac{%
(1-\lambda )^{2}}{\alpha q^{2}};$ 
\item[(iv)] $A_{1}(\eta )=A_{2}(\bar{\phi}^{-1}(\eta )),$ where $\bar{\phi}%
^{-1}$ is the inverse of $\bar{\phi}$ on $(0,\delta ]$ for some $\delta >0.$
\end{itemize}

Then the potential operator $K^{\alpha }$ is $\left( L_{\theta _{1},A_{1}}^{p),\lambda)}(X,\mu)_{N_0} - L_{\theta _{2},A_{2}}^{q),\lambda)}(X,\mu)_{N_0 \overline a }\right)$-bounded.
\end{theorem}

\begin{proof}
The proof follows the same lines as Theorem \ref{theo:boundedness_of_riesz}.
\end{proof}

\subsection*{Acknowledgment}

 The first and second authors were  partially supported by the Shota Rustaveli National Science Foundation Grant (Project No. GNSF/ST09\_ 23\_ 3-100). The third author was partially supported by \textsl{Funda\c c\~ao para a Ci\^encia e a Tecnologia} (FCT), \textsf{Grant SFRH/BPD/63085/2009}, Portugal and by Pontificia Universidad Javeriana.

\end{document}